\documentclass[12pt]{amsart}
\calclayout
\usepackage{amssymb}
\usepackage{graphics}
\usepackage{epsfig, color}
\newcommand\RR{\mathbb R}

\newcommand\K{\mathcal K}

\newcommand\T{\mathcal T}
\newcommand\E{\mathcal E}

\renewcommand\O{\mathcal O}

\renewcommand\div{\operatorname{div}}

\newcommand\x{\times}

\newcommand\lbra{[\![}  
\newcommand\rbra{]\!]}

\newcommand\lbrac{\,[\!\!\!\{} 
\newcommand\rbrac{\}\!\!\!]\,}

\renewcommand\ll{|\kern-2pt|\kern-2pt|}

\renewcommand\d{\mathrm{d}}

\numberwithin{equation}{section}
\theoremstyle{plain}
\newtheorem{thm}{Theorem}

\newtheorem{lem}[thm]{Lemma}

\numberwithin{thm}{section}

\theoremstyle{remark}

\def\stack#1#2#3{\rlap{#1}\lower#3\hbox{#2}}

\def\twiddlespace{1.2truept}

\def\dtwiddle{\displaystyle\sim}
\def\ttwiddle{\textstyle\sim}
\def\stwiddle{\scriptstyle\sim}
\def\sstwiddle{\scriptscriptstyle\sim}

\def\doubledtwiddle{\stack{$\dtwiddle$}{$\dtwiddle$}{\twiddlespace}}
\def\doublettwiddle{\stack{$\ttwiddle$}{$\ttwiddle$}{\twiddlespace}}
\def\doublestwiddle{\stack{$\stwiddle$}{$\stwiddle$}{\twiddlespace}}
\def\doublesstwiddle{\stack{$\sstwiddle$}{$\sstwiddle$}{\twiddlespace}}

\def\tripledtwiddle{\stack{$\dtwiddle$}{$\doubledtwiddle$}{\twiddlespace}}
\def\triplettwiddle{\stack{$\ttwiddle$}{$\doublettwiddle$}{\twiddlespace}}
\def\triplestwiddle{\stack{$\stwiddle$}{$\doublestwiddle$}{\twiddlespace}}
\def\triplesstwiddle{\stack{$\sstwiddle$}{$\doublesstwiddle$}{\twiddlespace}}

\def\quadrupledtwiddle{\stack{$\dtwiddle$}{$\tripledtwiddle$}{\twiddlespace}}
\def\quadruplettwiddle{\stack{$\ttwiddle$}{$\triplettwiddle$}{\twiddlespace}}
\def\quadruplestwiddle{\stack{$\stwiddle$}{$\triplestwiddle$}{\twiddlespace}}
\def\quadruplesstwiddle{\stack{$\sstwiddle$}{$\triplesstwiddle$}{\twiddlespace}}\def\quadru
pletwiddle{{\mathchoice{\quadrupledtwiddle}{\quadruplettwiddle}%

{\quadruplestwiddle}{\quadruplesstwiddle}}}

\catcode`\!=11
\newcommand\!a{{\boldsymbol a}}
\newcommand\!b{{\boldsymbol b}}
\newcommand\!c{{\boldsymbol c}}
\newcommand\!d{{\boldsymbol d}}
\newcommand\!e{{\boldsymbol e}}
\newcommand\!f{{\boldsymbol f}}
\newcommand\!g{{\boldsymbol g}}
\newcommand\!h{{\boldsymbol h}}
\newcommand\!i{{\boldsymbol i}}
\newcommand\!j{{\boldsymbol j}}
\newcommand\!k{{\boldsymbol k}}
\newcommand\!l{{\boldsymbol l}}
\newcommand\!m{{\boldsymbol m}}
\newcommand\!n{{\boldsymbol n}}
\newcommand\!o{{\boldsymbol o}}
\newcommand\!p{{\boldsymbol p}}
\newcommand\!q{{\boldsymbol q}}
\newcommand\!r{{\boldsymbol r}}
\newcommand\!s{{\boldsymbol s}}
\newcommand\!t{{\boldsymbol t}}
\newcommand\!u{{\boldsymbol u}}
\newcommand\!v{{\boldsymbol v}}
\newcommand\!w{{\boldsymbol w}}
\newcommand\!x{{\boldsymbol x}}
\newcommand\!y{{\boldsymbol y}}
\newcommand\!z{{\boldsymbol z}}
\newcommand\!A{{\boldsymbol A}}
\newcommand\!B{{\boldsymbol B}}
\newcommand\!C{{\boldsymbol C}}
\newcommand\!D{{\boldsymbol D}}
\newcommand\!E{{\boldsymbol E}}
\newcommand\!F{{\boldsymbol F}}
\newcommand\!G{{\boldsymbol G}}
\newcommand\!H{{\boldsymbol H}}
\newcommand\!I{{\boldsymbol I}}
\newcommand\!J{{\boldsymbol J}}
\newcommand\!K{{\boldsymbol K}}
\newcommand\!L{{\boldsymbol L}}
\newcommand\!M{{\boldsymbol M}}
\newcommand\!N{{\boldsymbol N}}
\newcommand\!O{{\boldsymbol O}}
\newcommand\!P{{\boldsymbol P}}
\newcommand\!Q{{\boldsymbol Q}}
\newcommand\!R{{\boldsymbol R}}
\newcommand\!S{{\boldsymbol S}}
\newcommand\!T{{\boldsymbol T}}
\newcommand\!U{{\boldsymbol U}}
\newcommand\!V{{\boldsymbol V}}
\newcommand\!W{{\boldsymbol W}}
\newcommand\!X{{\boldsymbol X}}
\newcommand\!Y{{\boldsymbol Y}}
\newcommand\!Z{{\boldsymbol Z}}
\newcommand\!alpha{{\boldsymbol\alpha}}
\newcommand\!beta{{\boldsymbol\beta}}
\newcommand\!gamma{{\boldsymbol\gamma}}
\newcommand\!delta{{\boldsymbol\delta}}
\newcommand\!epsilon{{\boldsymbol\epsilon}}
\newcommand\!zeta{{\boldsymbol\zeta}}
\newcommand\!eta{{\boldsymbol\eta}}
\newcommand\!theta{{\boldsymbol\theta}}
\newcommand\!iota{{\boldsymbol\iota}}
\newcommand\!kappa{{\boldsymbol\kappa}}
\newcommand\!lambda{{\boldsymbol\lambda}}
\newcommand\!mu{{\boldsymbol\mu}}
\newcommand\!nu{{\boldsymbol\nu}}
\newcommand\!xi{{\boldsymbol\xi}}
\newcommand\!pi{{\boldsymbol\pi}}
\newcommand\!rho{{\boldsymbol\rho}}
\newcommand\!sigma{{\boldsymbol\sigma}}
\newcommand\!tau{{\boldsymbol\tau}}
\newcommand\!upsilon{{\boldsymbol\upsilon}}
\newcommand\!phi{{\boldsymbol\phi}}
\newcommand\!chi{{\boldsymbol\chi}}
\newcommand\!psi{{\boldsymbol\psi}}
\newcommand\!omega{{\boldsymbol\omega}}
\newcommand\!varepsilon{{\boldsymbol\varepsilon}}
\newcommand\!vartheta{{\boldsymbol\vartheta}}
\newcommand\!varpi{{\boldsymbol\varpi}}
\newcommand\!varrho{{\boldsymbol\varrho}}
\newcommand\!varsigma{{\boldsymbol\varsigma}}
\newcommand\!varphi{{\boldsymbol\varphi}}
\newcommand\!Gamma{{\boldsymbol\Gamma}}
\newcommand\!Delta{{\boldsymbol\Delta}}
\newcommand\!Theta{{\boldsymbol\Theta}}
\newcommand\!Lambda{{\boldsymbol\Lambda}}
\newcommand\!Xi{{\boldsymbol\Xi}}
\newcommand\!Pi{{\boldsymbol\Pi}}
\newcommand\!Sigma{{\boldsymbol\Omega\eigma}}
\newcommand\!Upsilon{{\boldsymbol\Upsilon}}
\newcommand\!Phi{{\boldsymbol\Phi}}
\newcommand\!Psi{{\boldsymbol\Psi}}
\newcommand\!Omega{{\boldsymbol\Omega}}

\begin{document}

\title [Compact embedding]
{Compact embedding
in the space\\ of piecewise $H^1$ functions}

\author{Sheng Zhang}
\thanks{Department of Mathematics, Wayne State University, Detroit, MI 48202 (\texttt{szhang@wayne.edu})}


\begin{abstract}
We prove a compact embedding theorem in a class of spaces of piecewise $H^1$ functions subordinated
to a class of shape regular, but not necessarily quasi-uniform triangulations of a polygonal domain.
This result generalizes the Rellich--Kondrachov theorem.
It is used to prove generalizations to piecewise functions
of nonstandard Poincar\'e--Friedrichs 
inequalities. It can be used to prove Korn inequalities
for piecewise functions associated with elastic shells.
\vspace{12pt}

\noindent{\sc Key words.} Piecewise $H^1$ functions, compact embedding,
Rellich--Kondrachov, Poincar\'e--Friedrichs inequality.
\newline \noindent{\sc Subject classification.} 65N30, 46E35, 74S05.
\end{abstract}
\maketitle

\section{Introduction}
Let $\Omega$ be a bounded polygonal domain in $\RR^2$.
Let $\T_h$ be a shape regular, but not necessarily quasi-uniform,
triangulation on $\Omega$.
We also use $\T_h$ to denote the set of all (open) triangular elements
of the partition.
Let $H^1_h$ be the space of piecewise $H^1$ functions subordinated to the triangulation
$\T_h$. A function in $H^1_h$ is independently defined on
every element $\tau\in\T_h$ on which it belongs to $H^1(\tau)$. This kind of space of
piecewise functions arises in analysis of discontinuous finite element methods.
It is desirable to generalize the Sobolev space theory to such spaces.
Poincar\'e--Friedrichs type inequalities have been proposed and proved in several
forms in the literature \cite{A-int, Brenner-1, Feng}. In this paper, we
prove a compact embedding theorem that generalizes the Rellich--Kondrachov theorem.
Such compact embedding theorem can
be used, together with a modified
compactness argument, to prove the aforementioned Poincar\'e--Friedrichs
type inequality for piecewise functions
in a general form, including those in nonstandard forms \cite{Dahmen}.
It seems necessary
to generalizing Korn's inequalities on curved surfaces \cite{Ciarlet}
to spaces of piecewise functions, which
is useful in analyzing
discontinuous Galerkin finite element methods
for  models of elastic shells.

Shape regularity of triangulations
is a crucial notion in this paper. It is worthwhile to recall its definition
here. Considering a triangle, we let $r$ and $R$ be the radii of its
largest inscribed circle and smallest circumscribed circle, respectively.
Then the ratio $R/r$ is called its shape regularity constant, or simply shape regularity.
For a triangulation, the maximum
of shape regularity constants of all its triangles
is called its shape regularity constant \cite{Ciarlet-FEM-book}, denoted by $\K$.
We will need to consider a (infinite) family of triangulations. For a family, the
{\em shape regularity} constant $\K$ is the supreme of all the shape regularity constants
of its triangulations.
For a given triangulation $\T_h$, we use $\Omega_h$ to denote the union of
all the open triangular elements,
and use $\E^0_h$ denote the set of all interior (open) edges and $\E^\partial_h$ all
boundary edges.
A function $u$ in $H^1_h$
is certainly
in $L^2(\Omega)$ (in which the function value on $\E^0_h$ has no significance).
On an edge $e\in\E^0_h$, a function $u$ has two different traces from the
two elements sharing $e$. We use $\lbra u\rbra$ to denoted the difference
of the two traces, which is the jump of $u$ over $e$.
We furnish the space
$H^1_h$ with the norm
\begin{equation}\label{H1hnorm}
\|u\|_{H^1_h}:=\left[\|u\|^2_{L^2(\Omega)}+\int_{\Omega_h}|\nabla u|^2
+\sum_{e\in\E^0_h}\frac{1}{|e|}\int_e\lbra u\rbra^2\right]^{1/2}.
\end{equation}
Here $|e|$ is the length of the edge $e$. The integrals are taken with respect to
Lebesgue measures of the integration domains.
Let $\{\T_{h_i}\}$ be a family of triangulations
with a certain regularity constant $\K$.
Let $u_{i}\in H^1_{h_i}$ be uniformly bounded sequence of piecewise $H^1$ functions, i.e.,
there is a constant $C$ such that
$\|u_i\|_{H^1_{h_i}}\le C$ for all $i$. Then there is a subsequence $\{u_{i_k}\}$ that
is convergent in $L^2(\Omega)$.
It is this compact embedding theorem on
families of spaces that is needed to prove, for example, a Poincar\'e--Friedrichs
inequality that
there is a constant $C$ that is dependent on $\Omega$ and dependent on the triangulation
$\T_h$ only through its shape regularity but otherwise
{\em independent of $\T_h$} such that
\begin{equation*}  
\|u\|_{L^2(\Omega)}\le C
\left[\int_{\Omega_h}|\nabla u|^2
+\sum_{e\in\E^0_h}\frac{1}{|e|}\int_e\lbra u\rbra^2
+\int_{\E^{\partial}_h}u^2\right]^{1/2}\ \ \forall\ u\in H^1_h(\Omega).
\end{equation*}
Such constant independence of triangulations is fundamental in numerical analysis.

The remaining of the paper is organized as follows. In Section~\ref{sec-trace}, we present a trace
theorem for piecewise $H^1$ functions. This result will be used several times.
In Section~\ref{sec-compact}
we prove the compact embedding theorem. In the last section, we provide a new proof
for Poincar\'e--Friedrichs type inequalities for piecewise functions.
Throughout the paper, $C$ will be a generic constant that may depend on the domain $\Omega$ and shape
regularity $\K$ of a triangle, of a triangulation, or of a family of triangulations. But otherwise,
the constant is independent of triangulations.

\section{A trace theorem}\label{sec-trace}
We first prove a trace theorem for piecewise $H^1$ functions.
This result itself is a generalization of a trace theorem of Sobolev space theory.
It will be used in proving the compact embedding theorem, and in proving a Poincar\'e--Friedrichs inequality by using the compact embedding theorem.
We will need the following trace theorem on an element \cite{A-int}.
\begin{lem}
Let $\tau$ be a triangle, and $e$ one of its edges. Then there is a
constant $C$ depending on the shape regularity of $\tau$ such that
\begin{equation}\label{trace}
\int_eu^2\le C\left[|e|^{-1}\int_{\tau}u^2+|e|\int_{\tau}|\nabla u|^2\right]
\ \ \forall\ u\in H^1(\tau).
\end{equation}
\end{lem}
\begin{thm}\label{tracetheorem}
Let $\T_h$ be a shape regular, but not necessarily quasi-uniform triangulation of $\Omega$. There
exists a constant $C$ dependent on $\Omega$ and the shape regularity
constant of $\T_h$, but otherwise independent of the triangulation such that
\begin{equation}\label{Omega-trace}
\|u\|_{L^2(\partial\Omega)}\le C \|u\|_{H^1_h}\ \ \forall\ u\in H^1_h.
\end{equation}
\end{thm}
\begin{proof}
Let $\!phi$ be a piecewise smooth
vector field on $\Omega$ whose normal component is continuous across any straight line segment, and
such that $\!phi\cdot\!n=1$ on $\partial\Omega$. (The piecewise smoothness of $\!phi$ is not
associated with the triangulation $\T_h$. A construction of such vector field is given below.)
On each element
$\tau\in\T_h$, we have
\begin{equation*}
\int_{\partial\tau}u^2\!phi\cdot\!n=\int_\tau \div(u^2\!phi)=
\int_\tau(2u\nabla u\cdot\!phi+u^2\div\!phi).
\end{equation*}
Summing up over all elements of $\T_h$, we get
\begin{equation*}
\int_{\partial\Omega}u^2=-\sum_{e\in\E^0_h}\int_e\lbra u^2\!phi\rbra
+\int_{\Omega_h}(2u\nabla u\cdot\!phi+u^2\div\!phi).
\end{equation*}
If $e$ is the border between the elements
$\tau_1$ and
$\tau_2$ with outward normals $\!n_1$ and $\!n_2$, then
$\lbra u^2\!phi\rbra=u^2_1\!phi_1\cdot\!n_1+u^2_2\!phi_2\cdot\!n_2$,
where $u_1$ and $u_2$ are restrictions of $u$
on $\tau_1$ and $\tau_2$, respectively.  It is noted that although $\!phi$ may be discontinuous
across $e$, it normal component is continuous, i.e., $\!phi_1\cdot\!n_1+\!phi_2\cdot\!n_2=0$.
On the edge $e$, we have
$|\lbra u^2\!phi\rbra|\le|\lbra u^2\rbra|\|\!phi\|_{0,\infty,\Omega}$.
Here, $|\lbra u^2\rbra|=|u_1^2-u_2^2|$. It is noted that $|\lbra u^2\rbra|=2|\lbra u\rbra\lbrac u\rbrac|$,
with $\lbrac u\rbrac=(u_1+u_2)/2$ being the average. We have
\begin{multline}\label{jump-on-e}
\int_e|\lbra u^2\!phi\rbra|\le
2|\!phi|_{0,\infty,\Omega}
\left[|e|^{-1}\int_e\lbra u\rbra^2\right]^{1/2}\left[|e|\int_e\lbrac u\rbrac^2\right]^{1/2}\\
\le C
|\!phi|_{0,\infty,\Omega}\left[|e|^{-1}\int_e\lbra u\rbra^2\right]^{1/2}
\left[\int_{\delta e}u^2+|e|^2\int_{\delta e}|\nabla u|^2\right]^{1/2}.
\end{multline}
Here, $C$ only depends on the shape regularity of $\tau_1$ and $\tau_2$.
We used $\delta e=\tau_1\cup\tau_2$ to denote the ``co-boundary''
of edge $e$,
and we used the trace
estimate \eqref{trace}.
It then follows from the Cauchy--Schwarz inequality that
\begin{equation*}
\|u\|^2_{L^2(\partial\Omega)}\le C(|\!phi|_{0,\infty,\Omega}+|\div\!phi|_{0,\infty,\Omega})
\left[\|u\|^2_{L^2(\Omega)}+\int_{\Omega_h}|\nabla u|^2
+\sum_{e\in\E^0_h}\frac{1}{|e|}\int_e\lbra u\rbra^2\right].
\end{equation*}
Here the constant $C$ only depends on the shape regularity of $\T_h$.
The dependence on $\Omega$ of the $C$ in \eqref{Omega-trace}
is hidden in the $\!phi$ in the above inequality.
\end{proof}
%
%
We describe a construction of the vector field $\!phi$ used in the proof.
On the $xy$-plane, we consider a triangle $OAB$ with the origin being its
vertex $O$. Let the distance from $O$ to the side $AB$ be $H$. Then the field
$\!psi(x, y)=\langle x, y\rangle/H$ has the property that $\!psi\cdot\!n=1$ on $AB$
and $\!psi\cdot\!n=0$ on $OA$ and $OB$. Also $|\!psi|_{0,\infty}=\max\{|OA|, |OB|\}/H$
and $\div\!psi=2/H$.
For each straight segment of $\partial\Omega$, we define a triangle
with the straight segment being a side
whose opposite vertex is in $\Omega$, then we define a vector field on this triangle
as on the triangle $OAB$ with $AB$ being the straight side.
We need to assure that all such triangles do not overlap. We then piece together
all these vector fields and fill up the remaining part of the domain by a zero
vector field. This defines the desired vector field used in the proof.

\section{Compact embedding in the space of piecewise $H^1$ functions}\label{sec-compact}
We again assume that $\T_h$
is an arbitrary triangulation with shape regularity $\K$.
For $\delta>0$, we define a boundary strip $\Omega_\delta$ of width $\O(\delta)$.
We let $\Omega_\delta^0=\Omega\setminus \overline{\Omega_\delta}$ be the interior domain.
The interior domain $\Omega^0_\delta$ has the property that if a point is in $\Omega^0_\delta$, then the
disk centered at the point with radius $\delta$ entirely lies in $\Omega$.
We first show that when the strip is thin, the $L^2(\Omega_\delta)$
norm of the restriction on $\Omega_\delta$ of a function in $H^1_h$ must be small.
\begin{lem}\label{bd-shift-theorem}
There is a constant $C$ depending on $\Omega$ and the shape regularity $\K$ of $\T_h$,
but otherwise independent
of $\T_h$, such that
\begin{equation}\label{stripe}
\int_{\Omega_{\delta}}u^2(x)\d x\le C\delta\|u\|^2_{H^1_h}\ \ \forall\ u\in H^1_h.
\end{equation}
Here $\Omega_\delta$ is a boundary strip of width $\delta$ attached to $\partial\Omega$.
\end{lem}
\begin{proof}
We choose a piecewise smooth a vector field $\!phi$ whose normal component is continuous
across any curve such that $\!phi=0$ on the inner boundary of $\Omega_{\delta}$, and
$\div\!phi=1$ and $|\!phi|\le C\delta$ on $\Omega_{\delta}$.
(A construction of such $\!phi$ is given below.)
We then extend $\!phi$ by zero onto the entire domain $\Omega$. The extended, still denoted by $\!phi$,
is a piecewise smooth vector field whose normal components
is continuous over any curve in $\Omega$.
We thus have
\begin{equation*} 
\int_{\Omega_{\delta}}u^2=
\int_{\Omega}u^2\div\!phi
=\sum_{\tau\in\T_h}\int_{\tau}u^2\div\!phi=
-\sum_{\tau\in\T_h}\int_{\tau}2u\nabla u\cdot\!phi+
\sum_{\tau\in\T_h}\int_{\partial\tau}u^2\!phi\cdot\!n.
\end{equation*}
The last term can be written as
\begin{equation*}
\sum_{\tau\in\T_h}\int_{\partial\tau}u^2\!phi\cdot\!n=
\int_{\partial\Omega}u^2\!phi\cdot\!n+\sum_{e\in\E^0_h}\int_e\lbra u^2\!phi\rbra.
\end{equation*}
Since the normal components of $\!phi$ is continuous
on edges in $\E^0_h$, we use the same argument as in the proof of Theorem~\ref{tracetheorem}, cf., \eqref{jump-on-e},
to get
\begin{equation*}
\sum_{e\in\E^0_h}\int_e|\lbra u^2\!phi\rbra|\le C
|\!phi|_{0,\infty,\Omega}\left[\|u\|^2_{L^2(\Omega)}+\sum_{\tau\in\T_h}\int_{\tau}h^2_\tau|\nabla u|^2
+\sum_{e\in\E^0_h}\frac{1}{|e|}\int_e\lbra u\rbra^2\right].
\end{equation*}
It then follows from Theorem~\ref{tracetheorem} that
\begin{equation} \label{bd-shift}
\int_{\Omega_{\delta}}u^2\le C
|\!phi|_{0,\infty,\Omega}\left[\|u\|^2_{L^2(\Omega)}+\int_{\Omega_h}|\nabla u|^2
+\sum_{e\in\E^0_h}\frac{1}{|e|}\int_e\lbra u\rbra^2\right].
\end{equation}
The proof is complete since $|\!phi|_{0,\infty, \Omega}\le C\delta$.
The constant $C$ depends on $\Omega$ in terms of its interior angles and exterior angles at convex and
concave vertexes, respectively.
\end{proof}

We describe a way to choose the boundary strip and construct the vector field $\!phi$ that was
used in the proof.
This field can be constructed by piecing together several special vector fields.
We need some vector fields on rectangles, wedges, and circular disks.
On the $xy$-plane, we consider the vertical rectangular strip $R=(0,\delta)\x (0, l)$.
On this strip, we consider $\!psi_R=\langle x, 0\rangle$. This vector
field satisfies the condition that $\div\!psi_R=1$, $\!psi_R=0$ on the left side,
$\!psi_R\cdot\!n=0$ on the top and bottom sides and the maximum of $|\!psi_R|$ is $\delta$ that is attained
on the right side.
On a wedge $W$ with its vertex at the origin, we consider the vector field
$\!psi_W=\langle x, y\rangle/2$.
This $\!psi_W$ satisfies the conditions that
$\div\!psi_W=1$, $\!psi_W\cdot\!n=0$ on the two sides of $W$, and $|\!psi_W|=\rho/2$
at a point in $W$ whose distant from the origin is $\rho$.
On a circular disk $C$ centered at the origin and of radius $\rho$, we consider the vector field
$\!psi_C=(1-\rho^2/r^2)\langle x, y\rangle/2$. Here $r=(x^2+y^2)^{1/2}$.
This vector field satisfies the condition that $\div\!psi_C=1$ on the disk except at the center
where it is singular. It points toward the center  every where. And it is zero on the boundary.
We use this field on a sector of the circle $C$, on the two radial sides of which
we have $\!psi_C\cdot\!n=0$.
\begin{figure}[!ht]
\centerline{\input{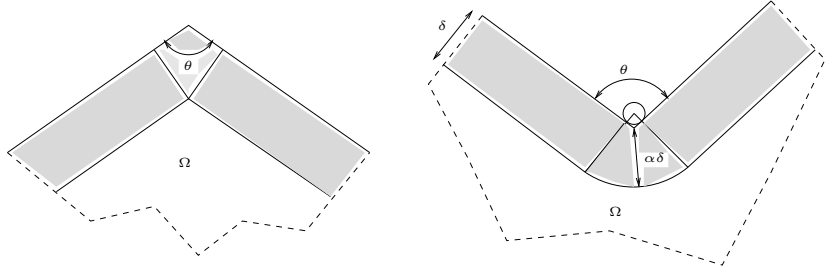}}
\caption{\label{boundarystrip}Boundary strip $\Omega_\delta$ near a convex vertex (left)
and a concave vertex (right).}
\end{figure}
With these special vector fields, we can then assemble the $\!phi$ on a boundary strip
$\Omega_\delta$. Along the interior side
of each straight segment of
$\partial\Omega$ we choose a uniform strip of thickness $\delta$. These strips overlap near
vertexes. If $\Omega$ is convex at a vertex, we introduce a wedge whose vertex is
at the intersection of the interior boundary of the uniform strips, and whose sides are orthogonal
to the meeting straight segments, see the left figure of Figure~\ref{boundarystrip}.
If $\Omega$ is concave at a vertex, we resolve it by using a circular sector, centered near the vertex
and outside of $\Omega$. The radius of the circle is slightly bigger than $\delta$
such that the arc is continuously connected to the
interior edges of the meeting strips, and the two radial sides are orthogonal to the meeting boundary
segments, see the right figure in Figure~\ref{boundarystrip}.
With such treatment of the vertexes, the boundary strip $\Omega_\delta$
is composed of
rectangular strips attaching to major portions of straight segments of
$\partial\Omega$, portion of wedges at convex vertexes, and portion of circular sectors
at concave vertexes.  See the shaded region in Figure~\ref{boundarystrip}.
We then transform $\!psi_R$, $\!psi_W$, and $\!psi_C$ to various parts of $\Omega_\delta$
and assemble a $\!phi$
on $\Omega_\delta$. The vector field $\!phi$ thus constructed is zero on the interior boundary
of $\Omega_\delta$. Its normal components are continuous across any curve, and
$\div\!phi=1$ on $\Omega_\delta$.
The thickness of $\Omega_\delta$ is the constant $\delta$ for the rectangular part.
It is maximized to $\delta/\sin\frac{\theta}2$ at a convex vertex. It is minimized to $\alpha\delta$
at the concave vertex, with $0<\alpha<1$, a value can be chosen as, for example, $1/2$. The norm $|\!phi|$
has a maximum $\delta/\sin\frac{\theta}2$ at a convex vertex with $\theta$ being the interior angle.
Thus when $\theta$ is small, $|\!phi|$ is big, and the estimate
\eqref{bd-shift} deteriorates.
The norm $|\!phi|$ also has a local maximum near a concave vertex.
It is bounded as
\begin{equation*}
|\!phi|\le\delta\frac{1-\alpha\sin\frac{\theta}{2}}{(1-\alpha)\sin\frac{\theta}{2}}.
\end{equation*}
When the exterior angle is small this maximum is big.
Also, one needs to choose $\alpha$ away from $1$ and $0$,
to maintain a moderate thickness of the strip and a reasonable bound for $|\!phi|$
which affect the estimate \eqref{bd-shift}.

We then prove that functions in $H^1_h$ are ``shift-continuous'' in $L^2$, as
stated in the following lemma. We extend a function $u\in H^1_h$
to a function $\tilde u$ on the whole $\RR^2$ by zero.
\begin{lem}\label{shift-continuity}
There is a constant $C$ depending on $\Omega$
and shape regularity $\K$ of $\T_h$, but otherwise independent of $\T_h$ such that
\begin{equation}\label{whole-shift}
\int_{\RR^2}[\tilde u(x+\rho)-\tilde u(x)]^2\d x\le C|\rho|\|u\|^2_{H^1_h}\ \ \forall\ u\in H^1_h.
\end{equation}
\end{lem}
\begin{proof}
Because for an element $\tau\in\T_h$, smooth functions are dense in $H^1(\tau)$, we only need
to prove \eqref{whole-shift}
for functions that are smooth on each element of $\T_h$. Let $u$ be such a piecewise smooth function.
Let $\rho$ be an arbitrary short vector. We take $\delta=|\rho|$ and choose a boundary strip $\Omega_\delta$.
The interior part $\Omega^0_{\delta}$ of the domain has the property that
if $x\in \Omega^0_{\delta}$ then the line segment $[x, x+\rho]\subset \Omega$. We have
\begin{equation*}
\int_{\RR^2}[\tilde u(x+\rho)-\tilde u(x)]^2\d x
=\int_{\Omega_\delta^0}[u(x+\rho)-u(x)]^2\d x
+\int_{\RR^2\setminus\Omega_\delta^0}[\tilde u(x+\rho)-\tilde u(x)]^2\d x.
\end{equation*}
Using Lemma~\ref{bd-shift-theorem}, we bound the second term as
\begin{equation}\label{bd-strip-est}
\int_{\RR^2\setminus\Omega_\delta^0}[\tilde u(x+\rho)-\tilde u(x)]^2\d x\le
\int_{\Omega_{2\delta}}u^2(x)\d x\le C|\rho|\|u\|^2_{H^1_h}.
\end{equation}
We then
focus on the first term.
This integral can be taken on an equal measure subset of
$\Omega^0_\delta$. This subset is obtained by removing a zero measure subset that
is composed of such point $x$:
$x$ or $x+\rho$ is on an open edge $e\in\E^0_h$,
or the closed line segment $[x, x+\rho]$ contains any vertex of $\T_h$,
or $[x, x+\rho]$ overlaps some edges of $\E^0_h$.
By such exclusion, for any $x$ in the remaining set, both the ends of $[x, x+\rho]$
are in the interior of some open triangular elements, and $[x, x+\rho]$ contains no
vertex. The restriction of $u$ on $[x, x+\rho]$ is a piecewise smooth one dimensional function,
which may have
a finite number of jumping points in $(x, x+\rho)$.
By the fundamental theorem
of calculus, we have
\begin{equation*}
u(x+\rho)-u(x)=\int_0^1\nabla u(x+t\rho)\cdot\rho \d t+\sum_{p\in
[x, x+\rho]\cap\E^0_h}\lbra u\rbra_p
\end{equation*}
Note that the integrand in the integral may make no sense at $t$, if
$x+t\rho\in\E^0_h$, where $u$ may jump. These points are excluded from the integration, where the
jumping effect is resolved in the second term.
On the segment $[x, x+\rho]$, $u$ may have a jump at $p\in [x, x+\rho]\cap\E^0_h$, which is denoted by
$\lbra u\rbra_p$ that is the value of $u$ from the $x$ side minus that from the $x+\rho$ side.
We thus have
\begin{equation*}
[u(x+\rho)-u(x)]^2\le
|\rho|^2\int_0^1|\nabla u(x+t\rho)|^2\d t+\left[\sum_{p\in
[x, x+\rho]\cap\E^0_h}\lbra u\rbra_p\right]^2.
\end{equation*}
When we take integral on $\Omega^0_\delta$ (minus the aforementioned zero-measure subset),
the first term is bounded as follows.
\begin{equation*}
\int_{\Omega^0_{\delta}}
|\rho|^2\int_0^1|\nabla u(x+t\rho)|^2\d t\d x=
|\rho|^2\int_0^1\int_{\Omega^0_{\delta}}|\nabla u(x+t\rho)|^2\d x\d t\\
\le
|\rho|^2\int_{\Omega_h}|\nabla u|^2\d x.
\end{equation*}

To estimate the jumping related second term, we write $\lbra u\rbra_p=|e|^{1/2}|e|^{-1/2}\lbra u\rbra_p$
if $p\in [x, x+\rho]\cap \E^0_h$ is on the edge $e$, and
use the Cauchy-Schwarz inequality to reach the following estimate.
\begin{equation*}
\left[\sum_{p\in
[x, x+\rho]\cap\E^0_h}\lbra u\rbra_p\right]^2\le
\left[\sum_{e\cap[x, x+\rho]\ne\emptyset}|e|^{-1}\lbra u\rbra^2_{e\cap[x, x+\rho]}\right]
\left[\sum_{e\cap[x, x+\rho]\ne\emptyset}|e|\right].
\end{equation*}
We show below that
there is a $C$, depending on the domain $\Omega$ and the
shape regularity $\K$ of $\T_h$ , such that
\begin{equation}\label{zigzag}
\sum_{e\cap[x, x+\rho]\ne\emptyset}|e|\le C.
\end{equation}
We then have
\begin{equation*}
\int_{\Omega^0_{\delta}}\left[\sum_{p\in
[x, x+\rho]\cap\E^0_h}\lbra u\rbra_p\right]^2\le
C\int_{\Omega^0_{\delta}}
\sum_{e\cap[x, x+\rho]\ne\emptyset}|e|^{-1}\lbra u\rbra^2_{e\cap[x, x+\rho]}\d x.
\end{equation*}
Every term in the right hand side is associated with a particular edge $e\in \E^0_h$. Each edge
$e\in\E^0_h$ is relevant to at most the points
in the parallelogram $\Omega_e$ in the Figure~\ref{shift}.
Thus, by changing the order of sum and integral, we have
\begin{multline*}
\int_{\Omega^0_{\delta}}
\sum_{e\cap[x, x+\rho]\ne\emptyset}|e|^{-1}\lbra u\rbra^2_{e\cap[x, x+\rho]}\d x
\le
\sum_{e\in\E^0_h}|e|^{-1}\int_{\Omega_e}\lbra u\rbra^2_{e\cap[x, x+\rho]}\d x\\
\le
\sum_{e\in\E^0_h}|e|^{-1}\sin\langle\rho, e\rangle|\rho|\int_{e}\lbra u\rbra^2
\le
|\rho|\sum_{e\in\E^0_h}|e|^{-1}\int_{e}\lbra u\rbra^2.
\end{multline*}
Here, $\langle\rho, e\rangle$ is the angle between the vector $\rho$ and the edge $e$.
\begin{figure}[!ht]
\centerline{\input{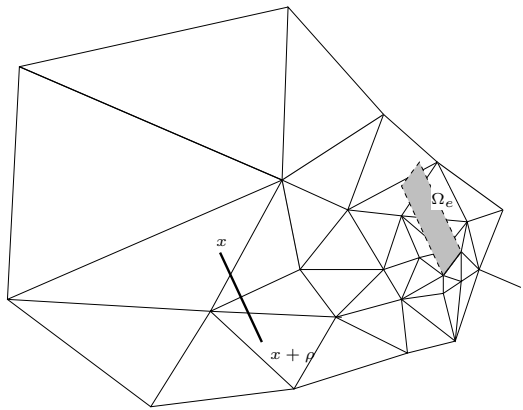}}
\caption{\label{shift}
A $\rho$ shift and $\Omega_e$ for an edge $e$.}
\end{figure}
Therefore, we have
\begin{equation*} 
\int_{\Omega^0_{\delta}}[u(x+\rho)-u(x)]^2dx\le
|\rho|^2|\nabla u|^2_{0, \Omega_h}+|\rho|\sum_{e\in\E^0_h}|e|^{-1}\int_e\lbra u\rbra^2.
\end{equation*}
Note that the second term may carry the smaller coefficient $|\rho|\max\{h,|\rho|\}$
such that the two terms
are closer in the order. But we do not need such refined estimates.
We thus proved
\begin{equation}\label{interior-shift}
\int_{\Omega^0_{\delta}}[u(x+\rho)-u(x)]^2\d x\le C|\rho|\|u\|_{H^1_h}\ \ \forall\ u\in H^1_h.
\end{equation}
The shift continuity \eqref{whole-shift} then follows from \eqref{interior-shift} and  \eqref{bd-strip-est}.
We have shown that the set of zero extended functions is shift continuous
in $L^2(\RR^2)$.
\end{proof}
We give a proof for the estimate \eqref{zigzag}.
Let $l$ be a straight line cutting through $\Omega$. Let $\T_h$ be a shape regular triangulation
with regularity constant $\K$. Then the sum of lengths of mesh line segments
intersecting $l$ is bounded independent of the triangulation. More specifically, 
we prove that there is a constant $C$, depending on the shape regularity $\K$, 
but otherwise independent of the triangulation $\T_h$
such that
\begin{equation}\label{zigzag1}
\sum_{e\in\E_h\text{ and }e\cap l\ne\emptyset}|e|\le C. 
\end{equation}
We shall use some facts that follow from the shape regularity assumption.
There is a minimum angle $\theta_\K$ among all angles of triangles of $\T_h$. The number of edges 
sharing a vertex is bounded by a constant $C$ that only depends on $\K$.
Let $e_1$ and $e_2$ be two edges 
sharing a 
vertex. There are constants $C_1$ and $C_2$ depending on $\K$ such that 
$|e_1|\le C_1|e_2|$ and $|e_2|\le C_2|e_1|$. 
Without loss of generality, we assume $l$ is horizontal. To simplify the argument, we also 
assume that $l$ does not pass any vertex. (This restriction can be removed by a slight modification
of the following argument.)
\begin{figure}[!ht]
\centerline{\input{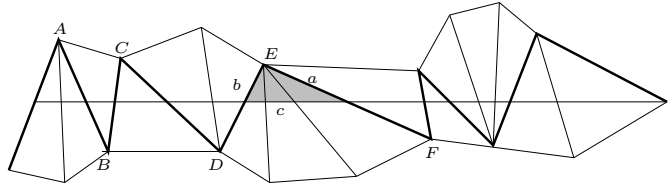}}
\caption{\label{cut-tri}A line cutting through the triangulation.}
\end{figure}
We first trim the set of intersecting edges  $\{e\in\E_h\text{ and }e\cap l\ne\emptyset\}$.
Consider the left most edge intersecting $l$, of which $A$ is the end vertex  
shared by some other edges intersecting $l$. If $A$ 
is above $l$, we examine all the edges intersecting $l$ and sharing $A$ in the counterclockwise order.
We discard all such edges but the last one that is $AB$ in Figure~\ref{cut-tri}.
(The next edge sharing $A$, as $AC$, does not intersect $l$.) The edge $BC$ must intersect $l$.
There could be other edges intersecting $l$ and sharing the vertex $B$.
We examine all the edges sharing $B$ and intersecting $l$ in the clockwise order. We discard
all but the last one. (It is $BC$ in the figure.) Now the vertex $C$ is in the same situation as $A$, and 
we can determine the edge $CD$ using the same rule as for $AB$. Then we determine $DE$, $EF$, and so forth.
We do the trimming all the way 
to the right end of $l$. This procedure touches all the edges 
intersecting $l$, by either trimming an edge off or keeping it.
The remaining edges constitute a continuous piecewise straight path
as represented by the thick line in the figure. We denote this set by $\E^l_h$. 
It follows from the aforementioned facts about the shape regular triangulation that there is 
a constant $C$, depending on $\K$ only, such that
\begin{equation*} 
\sum_{e\in\E_h\text{ and }e\cap l\ne\emptyset}|e|\le C\sum_{e\in\E^l_h}|e|. 
\end{equation*}
We consider a typical triangle bounded by $l$ and $\E^l_h$, as the shaded one 
in the figure, whose sides are $a$, $b$, and $c$. Let the angle $\angle DEF$ be denoted by $\theta$.
We have $\theta\ge\theta_\K$. Note that 
$c^2=a^2+b^2-2ab\cos\theta$. If $\theta$ is obtuse, then $a+b\le \sqrt 2 c$.
Otherwise, we have $c^2=(a^2+b^2)(1-\cos\theta)+(a-b)^2\cos\theta\le (a^2+b^2)(1-\cos\theta)$. Thus 
$a+b\le\sqrt{\frac2{1-\cos\theta}}c$. In any case, we have $a+b\le\sqrt{\frac2{1-\cos\theta_\K}}c$.
We thus proved 
\begin{equation*}
\sum_{e\in\E^l_h\cap\E^0_h}|e|\le\sqrt{\frac2{1-\cos\theta_\K}}|l\cap\Omega|.
\end{equation*}
From this, \eqref{zigzag1} follows. 
We can now prove the following compact embedding theorem.
\begin{thm}\label{uniformcompactembedding}
Let $\T_{h_i}$ be a (infinite) family of shape regular but not necessarily quasi-uniform
triangulations of the polygonal domain $\Omega$, with a shape regularity constant $\K$.
For each $i$, let $H^1_{h_i}$ be the space of piecewise $H^1$ functions, subordinated to the
triangulation $T_{h_i}$,
equipped with the norm \eqref{H1hnorm}.
Let $\{u_i\}$ be a bounded sequence such that $u_i\in H^1_{h_i}$ for each $i$. I.e.,
there is a constant $C$, such that $\|u_i\|_{H^1_{h_i}}\le C$ for all $i$.
Then, the sequence $\{u_i\}$ has a convergent subsequence in $L^2(\Omega)$.
\end{thm}
\begin{proof}
It follows from \eqref{whole-shift} that the sequence $\{u_i\}$ is a shift-continuous
subset of $L^2(\Omega)$.
The result then follows from the well-known condition
for a subset of $L^2(\Omega)$ to be compact, see Theorem~2.12 in \cite{Adams}, for example.
\end{proof}

\section{Poincar\'e--Friedrichs type inequalities for piecewise $H^1$ functions}
Poincar\'e--Friedrichs type inequalities have been generalized to spaces of
piecewise $H^1$ functions.
Several variants and proof methods
of such inequalities can be found, for example, in \cite{A-int, Brenner-1, Feng}.
We provide an alternative proof for this kind of inequalities by using the compactness result
of the previous section, to demonstrate how to use the compact embedding theorem
to obtain estimates that are uniformly valid with respect to triangulations.
This method is a modification of the classical methods of proving
some of the Poincar\'e--Friedrichs inequality based on compactness argument.

\begin{thm}\label{general-Poincare-thm}
Let $\T_h$ be a shape regular, but not necessarily quasi-uniform triangulation of the
polygon $\Omega$.
Let $f$ be a semi-norm
defined on the space $H^1_h$ that satisfies two conditions.

1) There is a constant $C$
that only depends on $\Omega$ and the shape regularity of $\T_h$, otherwise it is independent of $\T_h$, such that
\begin{equation}\label{fbound}
f(u)\le C\|u\|_{H^1_h}\ \ \forall\ u\in H^1_h.
\end{equation}

2) For any constant $c$, $f(c)=0$ if and only if $c=0$.

Then, there exists a constant $C$ depending only on the domain $\Omega$ and the shape
regularity constant of $\T_h$, but otherwise independent of $\T_h$,
such that
\begin{equation}\label{pw-general-Poincare}
\|u\|^2_{L^2(\Omega)}\le C\left[\int_{\Omega_h}|\nabla u|^2
+\sum_{e\in\E^0_h}\frac{1}{|e|}\int_e\lbra u\rbra^2
+f^2(u)\right]\ \ \forall\ u\in H^1_h.
\end{equation}
\end{thm}
A major point here is that the constant $C$ only depends on the shape regularity of $\T_h$.
Had one only considered a particular  triangulation, such inequality would
easily follow from the
Rellich--Kondrachov
compact embedding theorem \cite{Sobolev}
and Peetre's lemma (Theorem 2.1, page 18 in \cite{Raviart}).
But this argument can not establish the independence of the $C$ on the triangulation.
To establish the independence of $C$ of the triangulation,
we need to consider the entire family of all triangulations that
has a finite shape regularity constant as a family.
We need to modify the arguments of proving the Peetre's lemma to handle the family of spaces.
\begin{proof}[Proof of Theorem~\ref{general-Poincare-thm}]
The inequality \eqref{pw-general-Poincare} is equivalent to
\begin{equation}\label{general-Poincare-proof}
\|u\|^2_{H^1_h}\le C\left[\int_{\Omega_h}|\nabla u|^2
+\sum_{e\in\E^0_h}\frac{1}{|e|}\int_e\lbra u\rbra^2
+f^2(u)\right]\ \ \forall\ u\in H^1_h.
\end{equation}
If there is not a constant only depending on the shape regularity of $\T_h$, such that this inequality
holds, then there is a sequence of
shape regular triangulations $\T_{h_i}$ with a common shape regularity constant, and a sequence of functions
$u_i\in H^1_{h_i}$ such that

\begin{equation}\label{uis1}
\|u_i\|_{H^1_{h_i}}=1
\end{equation}and
\begin{equation}\label{to0}
\int_{\Omega_{h_i}}|\nabla u_i|^2
+\sum_{e\in\E^0_{h_i}}\frac{1}{|e|}\int_e\lbra u_i\rbra^2
+f^2(u_i)\to 0\ \ (i\to \infty).
\end{equation}
According to Theorem~\ref{uniformcompactembedding}, this sequence of functions has a convergent subsequence, still denoted by $u_i$,
in $L^2(\Omega)$. We let the limit be $u_0\in L^2(\Omega)$. We claim that this $u_0$ is a constant, and
the constant is zero.
We show this by verifying that the weak derivatives of $u_0$ is zero.
For a compactly supported smooth function $\phi$, we
have
\begin{equation*}
\int_{\Omega}u_0\partial_1\phi=\lim_{i\to\infty}\int_{\Omega_{h_i}}u_i\partial_1\phi.
\end{equation*}
For an $i$, by using integration by parts on each triangle,
we write the above right hand side as
\begin{equation*}
\int_{\Omega_{h_i}}u_i\partial_1\phi=
-\sum_{\tau\in\T_{h_i}}\int_\tau\partial_1 u_i\phi+\sum_{e\in\E^0_{h_i}}\int\lbra u_i\rbra_{n_1}\phi.
\end{equation*}
Here, on an edge $e$, $\lbra u\rbra_{n_1}=u_+n_{1+}+u_-n_{1-}$ if $e$ is shared by $\tau_+$ and $\tau_-$
and $n_{1+}$ is the first component of the unit outward normal of $\tau_+$.
Using Lemma~\ref{trace} to $\phi$, it follows that
\begin{equation*}
\left|\int_{\Omega_{h_i}}u_i\partial_1\phi\right|
\le
C\left[\int_{\Omega_{h_i}}|\nabla u_i|^2
+\sum_{e\in\E^0_{h_i}}\frac{1}{|e|}\int_e\lbra u_i\rbra^2\right]^{1/2}\|\phi\|_{H^1(\Omega)}.
\end{equation*}
In view of \eqref{to0}, we see that the weak derivative $\partial_1u_0$ is zero.
For the same reason, $\partial_2u_0=0$.
Thus $u_0$ is a constant.
We see from \eqref{to0} that $\|u_i-u_0\|_{H^1_{h_i}}\to 0$. Thus, by  the condition \eqref{fbound}
$f(u_0-u_i)\le C\|u_0-u_i\|_{H^1_{h_i}}\to 0$ as $i\to\infty$.
It follows from
$f(u_0)\le f(u_i)+f(u_0-u_i)\ \ \forall\ i$
that $f(u_0)=0$. Therefore, $u_0=0$. Thus,
$\|u_i\|_{H^1_{h_i}}\to 0$ when $i\to \infty$. This is contradict to \eqref{uis1}.
\end{proof}

In the classical Poincar\'e--Friedrichs inequalities, typical forms of the semi-norm
$f(u)$ are
\begin{equation*}
f_1(u)=\left[\int_\Gamma u^2\right]^{1/2},\quad
f_2(u)=\left|\int_{\Gamma}u\right|,\quad \text{or}\quad
f_3(u)=\left|\int_{\omega}u\right|.
\end{equation*}
Here $\Gamma$ is, or a portion of, the boundary $\partial\Omega$,
(which could be a segment with positive length of any curve
in $\Omega$,) and $\omega$ is, or a sub-domain of, $\Omega$.
Every one of these can be put in the position of the $f$ in Theorem~\ref{general-Poincare-thm}
to obtain Poincar\'e--Friedrichs inequalities for piecewise $H^1$ functions \cite{A-int, Brenner-1}.
To see the validity of the inequality \eqref{pw-general-Poincare} for each of these semi-norms,
one needs to verify the
uniform boundedness condition \eqref{fbound},
in which the constant $C$ is only allowed to depend on the
shape regularity of $\T_h$ and $\Omega$, and show that for any constant $c$,
$f(c)=0$ if and only if $c=0$. This latter condition is obviously
met by all of them. For $f_1$, the uniform boundedness condition follows from
the trace theorem, see \eqref{Omega-trace}. For $f_2$, the condition follows from the H\"older inequality
on $\Gamma$ and the condition for $f_1$. For $f_3$, it follows from the H\"older
inequality that $f_3(u)\le|\omega|^{1/2}\|u\|_{L^2(\omega)}\le|\Omega|^{1/2}\|u\|_{H^1_h}$.

By using a similar modification of the compactness argument, one can 
generalize the Korn inequality for shells (see \cite{Ciarlet}) 
to the space of piecewise functions, which seems useful 
for the analysis of discontinuous finite element methods for shells.


\bibliographystyle{plain}

\begin{thebibliography}{10}

\bibitem{Adams}
R.A. Adams,
\newblock{Sobolev spaces},
\newblock{\em Academic Press}, New York, 1975.

\bibitem{A-int}
D.N.~Arnold,
\newblock{An interior penalty finite element method with discontinuous
elements},
\newblock{\em SIAM J. Numer. Anal.}, 19 (1982), pp. 742-760.



\bibitem{Ciarlet}
M. Bernadou, P.G. Ciarlet, B. Miara,
Existence theorems for two-dimensional linear shell theories,
{\em J. Elasticity}, 34 (1994), pp. 111-138.



\bibitem{Brenner-1}
S.C. Brenner,
\newblock{Poincar\'e--Friedrichs inequalities for piecewise $H^1$ functions},
\newblock{\em SIAM J. Numer. Anal.}, 41(2003) pp. 306-324.

\bibitem{Brenner-2}
S.C. Brenner,
\newblock{Korn's inequalities for piecewise $H^1$ vector fields},
\newblock{\em Math. Comp.}, 73(2003) pp. 1067-1087.


\bibitem{Ciarlet-FEM-book}
P.G. Ciarlet,
\newblock{The finite element method for elliptic problems},
\newblock{\em North-Holland}, 1978.

\bibitem{Dahmen}
W. Dahmen, B. Faermann, I. G. Graham, W. Hackbusch, S. A. Sauter,
Inverse inequalities for non-quasiuniform meshes and applications to the mortar element method,
\newblock{\em Math. Comp.}, 73(2003) pp. 1107-1138.

\bibitem{Feng}
K. Feng,
\newblock{On the theory of discontinuous finite elements},
\newblock{\em Math. Numer. Sinica}, 4 (1979), pp. 378-385.

\bibitem{Raviart}
V. Girault, P-A. Raviart,
\newblock{Finite element methods for Navier--Stokes equations,  theory and algorithms},
\newblock{\em Springer-Verlag}, 1986.



\bibitem{Sobolev}
S.L. Sobolev,
\newblock{Some applications of functional analysis in mathematical physics},
{\em American Mathematical Society}, Providence, Rhode Island, 1991.


\end{thebibliography}

\end{document}